\documentclass[11pt]{article}
\usepackage[margin=1in]{geometry}
\usepackage[T1]{fontenc}
\usepackage{lmodern}
\usepackage{amsmath,amssymb,amsthm,mathtools}
\usepackage{microtype}
\usepackage[colorlinks=true,linkcolor=blue,citecolor=blue,urlcolor=blue]{hyperref}

\newcommand{\R}{\mathbb R}
\newcommand{\Sph}{\mathbb S}
\newcommand{\Var}{\operatorname{Var}}
\newcommand{\Ent}{\operatorname{Ent}}
\newcommand{\diag}{\operatorname{diag}}
\newcommand{\Id}{\mathrm{Id}}
\newcommand{\ip}[2]{\langle #1,#2\rangle}
\newcommand{\gradS}{\nabla_{\!S}}
\newcommand{\HessS}{\nabla_{\!S}^{2}}
\theoremstyle{plain}
\newtheorem{theorem}{Theorem}[section]
\newtheorem{proposition}[theorem]{Proposition}
\newtheorem{lemma}[theorem]{Lemma}
\newtheorem{corollary}[theorem]{Corollary}
\theoremstyle{remark}
\newtheorem{remark}[theorem]{Remark}
\newtheorem{example}[theorem]{Example}

\title{Blaschke--Santal\'o inequalities with respect to costs}
\author{A.~V.~Kolesnikov, Galyna V.~Livshyts and Liran Rotem}
\date{Revised October 8, 2026}

\begin{document}
\maketitle

\begin{abstract}
We study functional Blaschke--Santal\'o inequalities in which the inner product
is replaced by a cost. An inequality on each orthant, combined with Steiner
symmetrization, gives a sufficient condition for a cost inequality to hold for
even convex functions. We apply this principle to a multiplicative cost
defined on all orthants. Its second variation gives a weighted improvement
of the Brascamp--Lieb inequality in the directions for which differentiation
is justified. We also prove an $L^p$ version of the Fradelizi--Meyer theorem
from a weighted volume-product inequality and derive the corresponding radial
cost and Brascamp--Lieb inequalities. Explicit examples include anisotropic
power and Cauchy densities, with dilation equality cases and a sharp
obstruction below the coordinate exponent two. For the spherical
logarithmic cost we compute an exact entropy--logdet deficit decomposition
and the second-order deficit transverse to ellipsoids. We also explain the
transport-entropy formulation and distinguish the restricted spectral
question from the scalar $(B)$ property.
\end{abstract}

\section{Introduction and main results}

The Blaschke--Santal\'o inequality states that every origin-symmetric convex
body $K\subset\R^n$ satisfies
\begin{equation}\label{BS-bodies}
 |K|\,|K^\circ|\le |B_2^n|^2,
\end{equation}
where
\[
 K^\circ=\{y:\ \ip{x}{y}\le1\text{ for every }x\in K\}
\]
is the polar body. The inequality was proved by Blaschke in dimensions two
and three and by Santal\'o in arbitrary dimension
\cite{Blaschke,Santalo}.

Equality holds precisely for ellipsoids. The volume product is unchanged
under invertible linear maps, and this affine invariance is an important
feature of the classical inequality. Symmetrization proofs go back to
Saint Raymond \cite{SaintRaymond}; Meyer and Pajor \cite{MeyerPajor}
developed the Steiner symmetrization approach. Campi and Gronchi
\cite{CampiGronchi} and Meyer and Reisner \cite{MeyerReisner} studied the
behavior of polar volumes along shadow systems, giving a broader setting
for these arguments. Bianchi and Kelly \cite{BianchiKelly} later gave a
Fourier-analytic proof. For an account of the volume product, its history,
and its connections with other parts of mathematics, we refer to the
survey of Fradelizi, Meyer and Zvavitch \cite{FMZ}.

Ball \cite{Ball} discovered the functional inequality
\begin{equation}\label{BSBall-G}
 \int_{\R^n}e^{-\phi(x)}\,dx
 \int_{\R^n}e^{-\phi^*(y)}\,dy\le(2\pi)^n,
\end{equation}
where $\phi$ is even and
\[
 \phi^*(y)=\sup_x\{\ip{x}{y}-\phi(x)\}.
\]
Here and below, integral products are considered when their factors are
positive and finite; cases with a zero factor are understood separately.
Among lower semicontinuous convex functions, equality in
\eqref{BSBall-G} holds precisely for
$\phi(x)=a+\frac12\ip{Ax}{x}$, with $A$ positive definite.

Artstein-Avidan, Klartag and Milman \cite{AKM} introduced a Santal\'o
point for functions and extended the functional inequality beyond the
even case. Fradelizi and Meyer \cite{FM} obtained functional forms with
radial reference functions and several choices of centering. Their
results are particularly useful when one wants to replace the Gaussian
by another radial density. Lehec gave a direct proof by induction on the
dimension \cite{LehecDirect}, as well as a proof using partitions and the
logarithmic form of the Pr\'ekopa--Leindler inequality
\cite{LehecPartitions,Prekopa,Leindler}. Lin and Leng \cite{LinLeng}
extended the Steiner symmetrization proof of Meyer and Pajor to
functions.

Quantitative versions were obtained by B\"or\"oczky \cite{Bor} for convex
bodies and by Barthe, B\"or\"oczky and Fradelizi \cite{BBF} for the
functional inequalities. These results describe how a small deficit
forces a body or a function to be close, in an appropriate sense, to an
equality case. The reverse problem has also been studied in the
functional setting. Fradelizi and Meyer \cite{FMInverse,FMUnconditional}
proved inverse inequalities, including sharp results for unconditional
functions, and related them to Mahler's conjecture. Recent OpenAI
preprints announce general resolutions of the geometric Mahler problems;
see the dated discussion below. Our results concern upper bounds and
do not use these announced reverse inequalities.

There are several $L^p$ extensions of the Blaschke--Santal\'o inequality.
Lutwak and Zhang \cite{LutwakZhang} obtained inequalities involving
$L^p$ centroid bodies. More recently, Berndtsson, Mastrantonis and
Rubinstein \cite{BMR} introduced a different $L^p$ polarity based on the
Laplace transform and proved a Santal\'o inequality for symmetric convex
bodies. Their polarity approaches the usual one as $p\to\infty$.
Mastrantonis \cite{Mastrantonis} removed the symmetry assumption by
translating the body to its $L^p$ Santal\'o point.

A further approach comes from heat flow. Nakamura and Tsuji
\cite{NakamuraTsuji} proved that the functional volume product of an even
function is nondecreasing along the Fokker--Planck flow. Their argument
also gives a sharp Laplace transform inequality and an improvement of
reverse hypercontractivity for even functions. Cordero-Erausquin, Gozlan,
Nakamura and Tsuji \cite{CGNT} studied the interaction between Legendre
duality and heat flow, and showed that centered Gaussians are the only
local maximizers among even log-concave functions. In their recent paper,
Cordero-Erausquin, Fradelizi and Langharst \cite{CFL} extend the
Nakamura--Tsuji Laplace transform inequality to functions which need not
be even. They introduce a suitable Santal\'o point through minimization
over translations in a double Laplace transform, and use the
Fokker--Planck flow to prove the resulting inequality. Thus centering
continues to play an essential role in these extensions.

Optimal transport gives another interpretation of functional
Santal\'o inequalities. Fathi \cite{Fathi} obtained a sharp symmetrized
form of Talagrand's Gaussian transport-entropy inequality by dualizing
the functional Santal\'o inequality. Fradelizi, Gozlan, Sadovsky and
Zugmeyer \cite{FGSZ} developed this connection for further costs and
reference measures, including the spherical logarithmic cost considered
below. Kolesnikov and Werner \cite{KolesnikovWerner} introduced versions
with several functions and related them to Wasserstein barycenters.
Nakamura and Tsuji \cite{NakamuraTsujiDuality} proved the exponential-kernel
case of their conjectured inequality for even functions, through a
Gaussian extremality result for a generalized duality relation. They
also established the associated transport-entropy inequality for
barycenters of even measures.

Weighted volume products are especially close to the questions considered
here. For unconditional log-concave measures, the set inequality is
discussed in \cite{FMZ} and revisited by Milman and Yehudayoff \cite{MY},
who give a detailed account of the reduction from symmetric bodies to
unconditional ones. Colesanti, Kolesnikov, Livshyts and Rotem \cite{CKLR}
proved weighted functional Blaschke--Santal\'o inequalities and used
their variations to obtain strengthened Brascamp--Lieb inequalities.
Their symmetrization argument is one of the starting points of this
paper.

Langharst, Malliaris and Roysdon \cite{LMR} develop a general framework
of cost-Santal\'o inequalities on Polish measure spaces. Their
transference principle, \cite[Lemma~1.2]{LMR}, passes from inequalities
for sets to functional inequalities, and their Section~7 treats transport.
Their convention assigns the cost $-\ip{x}{y}$ to classical polarity;
our supremum convention below assigns $\ip{x}{y}$. The present paper
focuses on a coordinate Steiner criterion, explicit self-dual costs and
their second-variation matrices. In the homogeneous setting, our radial
argument also permits arbitrary measurable profiles $\rho$, whereas
the abstract transference lemma just cited assumes a nonincreasing
profile. We make the reverse indicator-function implication explicit in
Corollary~\ref{cor:FM-equivalence}.

Our goal is to replace the inner product by a cost $c(x,y)$. We use the convention
\begin{equation}\label{eq:def-c-transform}
 \phi^c(y)=\sup_x\{c(x,y)-\phi(x)\}.
\end{equation}
The supremum is always taken in the first variable. The cost may take the
value $-\infty$, in which case that pair contributes nothing to the
supremum.

Fix $p>1$. If $x$ and $y$ belong to the same open orthant, write
\[
 x^{1/p}y^{(p-1)/p}
 =\bigl(\operatorname{sgn}(x_i)|x_i|^{1/p}|y_i|^{(p-1)/p}\bigr)_{i=1}^n.
\]
Our first assumption is the lower bound
\begin{equation}\label{eq:cost-lower}
 c(x,y)\ge pV\bigl(x^{1/p}y^{(p-1)/p}\bigr)
 \qquad\text{when $x,y$ belong to the same orthant}.
\end{equation}
This gives the desired integral inequality when $e^{-\phi}$ has equal mass
in all orthants. To pass to even convex functions, we use a condition on the
cost which makes the dual integral increase under Steiner symmetrization.

For each coordinate hyperplane $H$, write $x=(X,s)$ and $y=(Y,t)$, with
$X,Y\in H$. The condition is
\begin{align}
 &c\left(\left(X,\frac{s_2-s_1}{2}\right),
              \left(\frac{Y_1+Y_2}{2},t\right)\right)\notag\\
 &\hspace{1cm}\le
 \frac12c((X,s_2),(Y_1,t))
 +\frac12c((X,s_1),(Y_2,-t)).\label{eq:SC}
\end{align}
It is required for all the displayed variables, including pairs in different
orthants. We refer to \eqref{eq:SC} as the Steiner condition.

\begin{theorem}\label{thm:main-cost}
Let $p>1$ and let $V:\R^n\to\R$ be measurable and unconditional, with
$0<\int e^{-V}<\infty$. Suppose that $c$ satisfies
\eqref{eq:cost-lower}, the Steiner condition in every coordinate, and
\[
 c(-x,-y)=c(x,y).
\]
Then every even lower semicontinuous proper convex coercive function $\phi$
with $\int e^{-\phi}>0$ satisfies
\begin{equation}\label{eq:main-cost}
 \left(\int e^{-\phi}\right)^{1/(p-1)}
 \int e^{-\phi^c/(p-1)}
 \le\left(\int e^{-V}\right)^{p/(p-1)}.
\end{equation}
The costs and transforms are assumed measurable.
\end{theorem}

The symmetry of the cost ensures that $\phi^c$ is even. More generally, the
proof only needs the integrals of $e^{-\phi^c/(p-1)}$ on the slices $t$ and
$-t$ to agree, at every stage of symmetrization.

We next describe a model to which the theorem applies. For arbitrary $x,y$,
set
\begin{equation}\label{eq:model-cost}
 c_{p,V}(x,y)=pV(z),\qquad
 z_i=
 \begin{cases}
 |x_i|^{1/p}|y_i|^{(p-1)/p},&x_i y_i>0,\\
 0,&x_i y_i\le0.
 \end{cases}
\end{equation}
Thus the geometric product in a coordinate is set to zero when the signs
disagree. On a common orthant, this is exactly the expression in
\eqref{eq:cost-lower}, because $V$ is unconditional.

The following convexity assumption is a convenient way to state the
Steiner condition for this cost. For each $j$, assume that
\begin{equation}\label{eq:power-convexity}
 (u_1,\ldots,u_n)\longmapsto
 V\bigl(u_1^{(p-1)/p},\ldots,u_{j-1}^{(p-1)/p},
        u_j^{1/p},u_{j+1}^{(p-1)/p},\ldots,u_n^{(p-1)/p}\bigr)
\end{equation}
is convex on $[0,\infty)^n$.

\begin{corollary}\label{cor:model-sharp}
Let $V:\R^n\to\R$ be unconditional, convex and coercive, and suppose that
\eqref{eq:power-convexity} holds. Then \eqref{eq:main-cost} holds with the
cost \eqref{eq:model-cost}. Moreover,
\begin{equation}\label{eq:model-self-dual}
 V^{c_{p,V}}=(p-1)V,
\end{equation}
so equality holds for $\phi=V$.
\end{corollary}

For example, \eqref{eq:power-convexity} holds for
$V(x)=\frac1r\sum_i|x_i|^r$ whenever
$r\ge\max\{p,p/(p-1)\}$. In particular, it holds when $r=p\ge2$.
It also holds for $V(x)=\|x\|_r^m/m$ with $m\ge r\ge2$, on taking $p=r$.
The argument uses the same reduction from even to unconditional functions
that appears in \cite{CKLR}, applied here to the cost itself.

The second variation of a cost inequality gives a Brascamp--Lieb inequality.
Recall that the usual Brascamp--Lieb inequality \cite{BL} is
\[
 \Var_\mu(f)\le
 \int\ip{(\nabla^2V)^{-1}\nabla f}{\nabla f}\,d\mu,
 \qquad d\mu=\frac{e^{-V}}{\int e^{-V}}\,dx.
\]
The improvement of variance inequalities under evenness is closely
related to stronger concavity properties of measures on symmetric
convex sets. This connection appears in the proof of the Gaussian
$(B)$ conjecture by Cordero-Erausquin, Fradelizi and Maurey \cite{CEFM}.
Cordero-Erausquin and Rotem \cite{CR} obtained sharp weighted inequalities
for rotationally invariant measures, which imply the $(B)$ and
Gardner--Zvavitch conjectures for rotationally invariant log-concave
measures. The cost formulation below gives a way to read the relevant
quadratic form directly from the second derivatives of the cost.

For the multiplicative cost, the matrix which appears instead is
\begin{equation}\label{eq:intro-mult-BL}
 \Var_\mu(f)\le
 \int\left\langle
 \left(\nabla^2V+\diag\left(\frac{\partial_iV}{x_i}\right)\right)^{-1}
 \nabla f,\nabla f\right\rangle\,d\mu.
\end{equation}
Under the hypotheses of Corollary~\ref{cor:model-sharp}, this follows for
even first variations for which the second variation is justified and the
displayed matrix is positive definite. The precise calculation is given in
Section~\ref{sec:differentiation-BL}. An extension to a larger class of test
functions requires approximation in the corresponding quadratic-form norm.
For unconditional functions, \eqref{eq:intro-mult-BL} follows directly from
the ordinary Brascamp--Lieb inequality after a logarithmic change of
variables. The additional term in the matrix is nonnegative, since $V$ is
unconditional and convex.

Our second family of inequalities comes from a volume-product inequality.
It gives the following extension of the Fradelizi--Meyer theorem.

\begin{theorem}\label{thm:intro-FM}
Let $p\ge r\ge2$ and $V(x)=\|x\|_r^p/p$. Let
$\rho:[0,\infty)\to[0,\infty)$ be measurable, and let $f,g$ be nonnegative
integrable functions, with $f$ even. If
\begin{equation}\label{eq:intro-FM-hyp}
 f(x)g(y)^{p-1}\le
 \rho\bigl(\ip{x}{\nabla V(y)}\bigr)^p
 \qquad\text{whenever }\ip{x}{\nabla V(y)}>0,
\end{equation}
then
\begin{equation}\label{eq:intro-FM-conclusion}
 \int f\left(\int g\right)^{p-1}
 \le\left(\int\rho(\|x\|_r^p)\,dx\right)^p.
\end{equation}
\end{theorem}

No monotonicity assumption on $\rho$ is needed. The case $p=r=2$ is the
Euclidean functional inequality of Fradelizi and Meyer. The input for the
stated range $p\ge r\ge2$ is the weighted Blaschke--Santal\'o inequality of
Colesanti, Kolesnikov, Livshyts and Rotem \cite{CKLR}.

Finally, we consider a cost on the sphere. Kolesnikov
\cite{KolesnikovSphere} studied transport functionals for the logarithmic
inner-product cost and their relation to the symmetric logarithmic
Minkowski problem. On $\Sph^{n-1}$, this cost is
\[
 c(x,y)=\log\ip{x}{y}\quad\text{if }\ip{x}{y}>0,
 \qquad c(x,y)=-\infty\quad\text{otherwise}.
\]
With $\sigma$ denoting normalized surface measure, the classical
Blaschke--Santal\'o inequality gives
\begin{equation}\label{eq:intro-spherical}
 \int_{\Sph^{n-1}}e^{-nf}\,d\sigma
 \int_{\Sph^{n-1}}e^{-nf^c}\,d\sigma\le1
\end{equation}
for even $f$. We compute its contact map, its second variation, and the
equation satisfied by smooth extremizers. We also recover
\begin{equation}\label{eq:intro-transport}
 K(\mu,\nu)\le
 \frac1n\Ent_\sigma(\mu)+\frac1n\Ent_\sigma(\nu),
\end{equation}
for even probability measures, where $K$ is the optimal transport cost for
$-\log\ip{x}{y}$ on the positive-inner-product region. The case
$\mu=\sigma$ was proved by Kolesnikov \cite{KolesnikovSphere}; the full
two-measure inequality \eqref{eq:intro-transport} is
\cite[Theorem~3.7]{FGSZ}.

The paper is organized as follows. Section~\ref{sec:orthant-cost} proves
the orthant inequality. Section~\ref{sec:steiner-extension} contains the
Steiner argument and the proof of the main theorem.
Section~\ref{sec:differentiation-BL} gives the second-variation calculation.
Section~\ref{sec:p-FM} proves the $L^p$ Fradelizi--Meyer equivalence
and its radial and transport consequences. Section~\ref{sec:examples}
collects explicit examples and sharpness calculations.
Section~\ref{sec:spherical-BS} treats the spherical cost and its deficits.
Section~\ref{sec:restriction} discusses restrictions to convex bodies
and their relation to dilation concavity.

\medskip
\noindent\textbf{Acknowledgement.}
Galyna V.~Livshyts and Liran Rotem are supported by NSF--BSF DMS-2247834.
Liran Rotem is grateful to the Lady Davis Fellowship at the Technion.

\medskip
\noindent\textbf{Statement of AI use.} The main ideas and results were developed by the authors before Spring 2026. AI was used for writing and polishing, and some minor technical advances. More recently, AI was prompted to add section 7 with the examples.

\paragraph{Recent developments (October 8, 2026).}
OpenAI's October 6 release \cite{OpenAIRelease,OpenAICollection} includes
a preprint announcing the logarithmic Brunn--Minkowski inequality in
every dimension and, by transfer, the scalar-dilation $(B)$ property
for even log-concave Radon measures \cite{OpenAILogBM}. The same
collection announces symmetric and nonsymmetric Mahler inequalities
(result family~087); the logarithmic Brunn--Minkowski result is
family~091. In view of the reported $(B)$ result, the question
in Section~\ref{sec:restriction} is framed as a stronger restricted
spectral estimate and a possible independent mechanism for dilation
concavity, rather than as a claim that the scalar $(B)$ problem remains
unresolved. Generally, this new influx of ideas should bring a lot of development to the topics developed in this paper, but we choose to separate the human-driven part, as presented in this note, from any future advances, which will be largely AI-driven.

\section{Preliminaries}

Unless another measure is specified, integrals on $\R^n$ are with respect
to Lebesgue measure. A function is even if $F(-x)=F(x)$ and unconditional
if it is invariant under every coordinate reflection. For
$\varepsilon\in\{-1,1\}^n$, let
\[
 Q_\varepsilon=\{x:\ \varepsilon_i x_i>0\text{ for all }i\}.
\]
We say that a nonnegative integrable function $h$ has equal orthant masses
if
\[
 \int_{Q_\varepsilon}h=2^{-n}\int h
 \qquad\text{for every }\varepsilon.
\]
In particular, an unconditional function has equal orthant masses.

A convex potential is taken to be lower semicontinuous and proper. It is
coercive if it tends to $+\infty$ as $|x|\to\infty$. We allow convex test
potentials to take the value $+\infty$, and define their cost transforms by
taking the supremum only over the points where the potential is finite.
All transforms which occur in integral statements are assumed measurable.
The reference potentials used in the model costs are finite-valued.

An unconditional convex function is nondecreasing in each coordinate on
$[0,\infty)^n$. Indeed, decreasing the absolute value of one coordinate
amounts to taking a convex combination of a point and its reflection.
Consequently, for $0<a<1$ and $x,y$ in a common orthant,
\begin{equation}\label{eq:geometric-convexity}
 V(x^a y^{1-a})
 \le V(a|x|+(1-a)|y|)
 \le aV(x)+(1-a)V(y),
\end{equation}
where the absolute values in the middle expression are coordinatewise.
Thus geometric convexity is automatic for the unconditional convex
potentials considered here.

We will use the Pr\'ekopa--Leindler inequality
\cite{Prekopa,Leindler}, including its multiplicative form on an orthant.
The latter follows by setting $x_i=e^{u_i}$ and including the Jacobian in
each of the three functions.

\section{The orthant-wise cost inequality}\label{sec:orthant-cost}

The first step does not require convexity or symmetry of the reference
potential. It also keeps track of the mass in each orthant.

\begin{theorem}\label{thm:orthant}
Let $p>1$, let $V$ be measurable with $e^{-V}$ integrable, and suppose
$c$ satisfies \eqref{eq:cost-lower}. If $e^{-\phi}$ has positive finite
integral and equal orthant masses, then
\begin{equation}\label{eq:orthant-bound}
 \left(\int e^{-\phi}\right)^{1/(p-1)}
 \int e^{-\phi^c/(p-1)}
 \le2^{n/(p-1)}\sum_\varepsilon
       \left(\int_{Q_\varepsilon}e^{-V}\right)^{p/(p-1)}.
\end{equation}
If $e^{-V}$ also has equal orthant masses, the right-hand side is
$(\int e^{-V})^{p/(p-1)}$.
\end{theorem}

\begin{proof}
Write $F=e^{-\phi}$, $G=e^{-\phi^c/(p-1)}$ and $H=e^{-V}$. By the
definition of the transform,
\[
 H(x^{1/p}y^{(p-1)/p})\ge F(x)^{1/p}G(y)^{(p-1)/p}
\]
on each common orthant. On the positive orthant, set $x=e^u$, $y=e^v$
and define
\[
 \widetilde F(u)=e^{\sum_i u_i}F(e^u),\qquad
 \widetilde G(v)=e^{\sum_i v_i}G(e^v),\qquad
 \widetilde H(w)=e^{\sum_i w_i}H(e^w).
\]
Then
\[
 \widetilde H\left(\frac{u}{p}+\frac{p-1}{p}v\right)
 \ge\widetilde F(u)^{1/p}\widetilde G(v)^{(p-1)/p}.
\]
Pr\'ekopa--Leindler gives
\[
 \int_{Q_+}H\ge
 \left(\int_{Q_+}F\right)^{1/p}
 \left(\int_{Q_+}G\right)^{(p-1)/p}.
\]
The same argument applies after a reflection to every other orthant. Hence
\[
 \int_{Q_\varepsilon}G\le
 \left(2^{-n}\int F\right)^{-1/(p-1)}
 \left(\int_{Q_\varepsilon}H\right)^{p/(p-1)}.
\]
Summing proves \eqref{eq:orthant-bound}. If $H$ also has equal orthant
masses, the powers of $2$ cancel. The argument applies by truncation when
the integrability of $G$ is not known in advance, and the resulting bound
then proves that $G$ is integrable.
\end{proof}

For the model cost \eqref{eq:model-cost}, unconditional convexity gives
the useful pointwise inequality
\begin{equation}\label{eq:model-support}
 c_{p,V}(x,y)\le V(x)+(p-1)V(y).
\end{equation}
Indeed, its vector $z$ is coordinatewise at most
$|x|^{1/p}|y|^{(p-1)/p}$, and \eqref{eq:geometric-convexity} applies.
Equality holds at $x=y$, including points on coordinate hyperplanes.
Taking the supremum in $x$ proves \eqref{eq:model-self-dual}.

It follows that the model inequality is sharp on the class of functions
with equal orthant masses, for every finite unconditional convex coercive
$V$. The additional assumption \eqref{eq:power-convexity} is used only to
pass to all even convex functions.

\section{Steiner symmetrization and the proof of the main theorem}
\label{sec:steiner-extension}

For a convex body $K$ and a coordinate hyperplane $H$, Steiner
symmetrization replaces each nonempty interval
$\{s:(X,s)\in K\}$ by the centered interval of the same length.
Empty fibers stay empty. We denote the resulting body by $S_HK$.
For a convex coercive potential $\phi$, define $S_H\phi$ by
\[
 \{S_H\phi\le a\}=S_H\{\phi\le a\}.
\]
Equivalently, symmetrize its epigraph in the same spatial direction.
This preserves convexity and the distribution of the sublevel-set
volumes. In particular,
\begin{equation}\label{eq:steiner-primal}
 \int e^{-S_H\phi}=\int e^{-\phi}.
\end{equation}
It also preserves central symmetry.

\begin{lemma}\label{lem:steiner}
Suppose $c$ satisfies \eqref{eq:SC} for $H$, and let $\phi$ be a convex
coercive potential. For $\beta>0$, assume
\[
 \int_H e^{-\beta\phi^c(Y,t)}\,dY
 =\int_H e^{-\beta\phi^c(Y,-t)}\,dY
 \qquad\text{for almost every }t.
\]
Then
\begin{equation}\label{eq:steiner-dual}
 \int e^{-\beta(S_H\phi)^c}\ge\int e^{-\beta\phi^c}.
\end{equation}
\end{lemma}

\begin{proof}
Fix $Y_1,Y_2,t$. If $S_H\phi(X,s)\le a$, the fiber of $\{\phi\le a\}$
over $X$ contains two points $s_1,s_2$ with $s=(s_2-s_1)/2$.
This uses the fact that the original fiber is an interval; the two points
need not have opposite signs. The definition of the transform gives
\[
 c((X,s_2),(Y_1,t))\le a+\phi^c(Y_1,t),\qquad
 c((X,s_1),(Y_2,-t))\le a+\phi^c(Y_2,-t).
\]
By \eqref{eq:SC},
\[
 c\left((X,s),\left(\frac{Y_1+Y_2}{2},t\right)\right)
 \le a+\frac{\phi^c(Y_1,t)+\phi^c(Y_2,-t)}2.
\]
Letting $a\downarrow S_H\phi(X,s)$ and taking the supremum over $(X,s)$
yields
\[
 (S_H\phi)^c\left(\frac{Y_1+Y_2}{2},t\right)
 \le\frac{\phi^c(Y_1,t)+\phi^c(Y_2,-t)}2.
\]
Apply Pr\'ekopa--Leindler on $H$ to the exponentials. This gives
\begin{align*}
 \int_H e^{-\beta(S_H\phi)^c(Y,t)}\,dY
 &\ge
 \left(\int_H e^{-\beta\phi^c(Y,t)}\,dY\right)^{1/2}
 \left(\int_H e^{-\beta\phi^c(Y,-t)}\,dY\right)^{1/2}\\
 &=\int_H e^{-\beta\phi^c(Y,t)}\,dY.
\end{align*}
Integration in $t$ proves the claim. Extended integral values can be
handled by monotone truncation in Pr\'ekopa--Leindler.
\end{proof}

\begin{proof}[Proof of Theorem~\ref{thm:main-cost}]
If $\phi$ is even and $c(-x,-y)=c(x,y)$, then $\phi^c$ is even. Thus its
slice integrals at $t$ and $-t$ agree after changing $Y$ to $-Y$.
Apply Lemma~\ref{lem:steiner} with $\beta=1/(p-1)$ successively in all
coordinate directions.

Symmetrization in one coordinate preserves reflection symmetry in the
other coordinates. Therefore one pass through the $n$ coordinates gives
an unconditional potential $\widetilde\phi$, with
\[
 \int e^{-\widetilde\phi}=\int e^{-\phi},\qquad
 \int e^{-\phi^c/(p-1)}\le
 \int e^{-\widetilde\phi^c/(p-1)}.
\]
The last potential has equal orthant masses, so Theorem~\ref{thm:orthant}
applies. Since $V$ is unconditional, its orthant masses are also equal,
and we obtain \eqref{eq:main-cost}.
\end{proof}

\begin{proof}[Proof of Corollary~\ref{cor:model-sharp}]
Only the Steiner condition remains to be checked. Fix the last coordinate
and first take $t>0$. For fixed $X$, define
\[
 F(a,b)=V\bigl(|X_1|^{1/p}b_1^{(p-1)/p},\ldots,
 |X_{n-1}|^{1/p}b_{n-1}^{(p-1)/p},t^{(p-1)/p}a^{1/p}\bigr)
\]
for $a\ge0$ and $b\in[0,\infty)^{n-1}$. Assumption
\eqref{eq:power-convexity}, with a rescaling of its arguments, says that
$F$ is convex. It is also nondecreasing in each variable.

Write $r_+=\max\{r,0\}$ and choose
$\varepsilon_i=\operatorname{sgn}(X_i)$ when $X_i\ne0$; its value is
irrelevant when $X_i=0$. The function
\[
 (s,Y)\longmapsto F\bigl(s_+,((\varepsilon_iY_i)_+)_{i<n}\bigr)
\]
is convex, as a composition of a convex coordinatewise nondecreasing
function with convex functions. Its value is $c_{p,V}((X,s),(Y,t))/p$.
Its value at $(-s_1,Y_2)$ is
$c_{p,V}((X,s_1),(Y_2,-t))/p$. Convexity now gives exactly
\eqref{eq:SC}. For $t<0$, reverse the last-coordinate signs; for $t=0$,
use continuity. The same argument works in every coordinate.

The cost is invariant under simultaneous coordinate reflections, so
Theorem~\ref{thm:main-cost} applies. Equation
\eqref{eq:model-self-dual}, already proved from \eqref{eq:model-support},
gives equality at $V$.
\end{proof}

\begin{remark}
Assumption \eqref{eq:power-convexity} can equivalently be written as the
Steiner inequality for the geometric products when all three products
belong to common orthants. To see this, fix $X$ and $t$, and regard the
inequality as midpoint convexity in $(s,Y)$ on the positive orthant.
Continuity turns midpoint convexity into convexity. The definition
\eqref{eq:model-cost} is what extends that condition to arbitrary signs.
\end{remark}

\begin{corollary}\label{cor:quadratic}
Let $A=\diag(a_1,\ldots,a_n)$ with $a_i>0$, and let
$c(x,y)=\ip{Ax}{y}$. Then every even convex coercive potential satisfies
\[
 \int e^{-\phi}\int e^{-\phi^c}
 \le\frac{(2\pi)^n}{\det A}.
\]
\end{corollary}

\begin{proof}
Take $p=2$ and $V(x)=\frac12\ip{Ax}{x}$. The lower bound
\eqref{eq:cost-lower} and the Steiner condition both hold with equality,
and the cost has the required symmetry. The Gaussian integral gives the
constant.
\end{proof}

\section{Differentiation and Brascamp--Lieb inequalities}
\label{sec:differentiation-BL}

\subsection{The general cost}

Fix $\alpha>0$ and consider
\begin{equation}\label{eq:J}
 J(\phi)=\alpha\log\int e^{-\phi}
            +\log\int e^{-\alpha\phi^c}.
\end{equation}
We first calculate its second variation without making a symmetry
assumption. Let $\phi_t$ be a two-sided smooth family, with
$\phi_0=V$, and write
\[
 f=\left.\frac{d}{dt}\right|_{0}\phi_t,
 \qquad k=\left.\frac{d^2}{dt^2}\right|_{0}\phi_t.
\]
In this subsection we assume that the integrals can be differentiated
twice. We also assume that, for almost every $y$, the maximum defining
$\phi_t^c(y)$ is attained at a unique point $T_t(y)$ which depends smoothly
on $t$. The cost and the potentials must be twice differentiable near
these contact points. Put $T=T_0$ and assume
\[
 M(y)=\nabla^2V(T(y))-D^2_{xx}c(T(y),y)>0.
\]

The first-order condition is
\[
 D_xc(T_t(y),y)=\nabla\phi_t(T_t(y)).
\]
Differentiating it gives
\[
 \dot T_0(y)=-M(y)^{-1}\nabla f(T(y)).
\]
On the other hand, differentiating
$\phi_t^c(y)=c(T_t(y),y)-\phi_t(T_t(y))$ cancels the terms involving
$\dot T_t$. Thus
\begin{align}
 \left.\frac{d}{dt}\right|_0\phi_t^c(y)&=-f(T(y)),\label{eq:c-first}\\
 \left.\frac{d^2}{dt^2}\right|_0\phi_t^c(y)
 &=-k(T(y))+
 \ip{M(y)^{-1}\nabla f(T(y))}{\nabla f(T(y))}.
 \label{eq:c-second}
\end{align}

Normalize the two measures at $t=0$:
\[
 d\mu=\frac{e^{-V}}{\int e^{-V}}\,dx,
 \qquad d\nu=\frac{e^{-\alpha V^c}}{\int e^{-\alpha V^c}}\,dy.
\]
The logarithmic derivatives of the first integral are
\[
 \left.\frac{d}{dt}\right|_0\log\int e^{-\phi_t}=-\int f\,d\mu,
 \qquad
 \left.\frac{d^2}{dt^2}\right|_0\log\int e^{-\phi_t}
 =\Var_\mu(f)-\int k\,d\mu.
\]
Using \eqref{eq:c-first} in the second integral gives
\begin{equation}\label{eq:J-first}
 J'(0)=\alpha\left(\int f\circ T\,d\nu-\int f\,d\mu\right).
\end{equation}
Consequently, criticality implies $T_\#\nu=\mu$ whenever the allowed
variations determine the measure. Even variations suffice if both
$T_\#\nu$ and $\mu$ are even.

Before making this identification, the second derivative is
\begin{align}
 J''(0)
 &=\alpha\Var_\mu(f)+\alpha^2\Var_\nu(f\circ T)
   -\alpha\int k\,d\mu+\alpha\int k\circ T\,d\nu\notag\\
 &\quad-\alpha\int
 \ip{M(y)^{-1}\nabla f(T(y))}{\nabla f(T(y))}\,d\nu(y).
 \label{eq:J-second}
\end{align}
If $T_\#\nu=\mu$, the terms containing $k$ cancel. We obtain
\begin{equation}\label{eq:J-second-transport}
 J''(0)=\alpha(1+\alpha)\Var_\mu(f)
 -\alpha\int
 \ip{M(y)^{-1}\nabla f(T(y))}{\nabla f(T(y))}\,d\nu(y).
\end{equation}

\begin{proposition}\label{prop:local-BL}
Suppose $J$ has a local maximum at $V$ along a family satisfying the
assumptions above. If
\[
 T(y)=y,\qquad \alpha V^c(y)=V(y)+C
\]
almost everywhere, then its first variation $f$ satisfies
\begin{equation}\label{eq:local-BL}
 \Var_\mu(f)\le\frac1{1+\alpha}
 \int\left\langle
 \bigl(\nabla^2V-D^2_{xx}c(x,x)\bigr)^{-1}\nabla f,
 \nabla f\right\rangle\,d\mu.
\end{equation}
\end{proposition}

\begin{proof}
The second identity gives $\nu=\mu$ after normalization. Since $T=\Id$,
\eqref{eq:J-second-transport} applies. The inequality $J''(0)\le0$ gives
the result.
\end{proof}

When the cost inequality is only known for even convex potentials, the
family in this proposition must stay in that class for both signs of
$t$. The inequality then holds for its even first variations, and passes
to their closure in $L^2(\mu)$ and in the energy on the right of
\eqref{eq:local-BL}.

\subsection{The multiplicative model cost}

Let $c=c_{p,V}$ be defined by \eqref{eq:model-cost}, and assume $V$ is
$C^2$ on the open orthants. Near a diagonal point $x=y$ with no zero
coordinates, the two vectors have matching signs. We may therefore use
the original geometric product to compute the derivatives, with the
signs restored inside $V$ by unconditionality.

Set $a=1/p$, $b=(p-1)/p$ and
$z_i=\operatorname{sgn}(x_i)|x_i|^a|y_i|^b$ locally. At $x=y$,
\[
 \partial_{x_i}z_i=\frac1p,
 \qquad
 \partial_{x_i}^2z_i=-\frac{p-1}{p^2x_i}.
\]
The chain rule gives
\begin{equation}\label{eq:model-Hessian}
 D^2_{xx}c_{p,V}(x,x)
 =\frac1p\nabla^2V(x)
  -\frac{p-1}{p}\diag\left(\frac{\partial_iV(x)}{x_i}\right).
\end{equation}
In particular,
\begin{equation}\label{eq:model-matrix}
 \nabla^2V-D^2_{xx}c_{p,V}(x,x)
 =\frac{p-1}{p}
  \left(\nabla^2V+\diag\left(\frac{\partial_iV}{x_i}\right)\right).
\end{equation}
Coordinate hyperplanes have Lebesgue measure zero, so no value of the
ratios there is needed in the integral inequality.

\begin{corollary}\label{cor:mult-BL}
Assume the hypotheses of Corollary~\ref{cor:model-sharp}, and suppose the
second variation is justified as in the preceding subsection, with the
unique contact point $T(y)=y$ almost everywhere. If the matrix in
\eqref{eq:intro-mult-BL} is positive definite, then every admissible even
first variation satisfies \eqref{eq:intro-mult-BL}.
\end{corollary}

\begin{proof}
Take $\alpha=1/(p-1)$. By \eqref{eq:model-self-dual},
$\alpha V^c=V$, and the cost inequality makes $V$ a maximizer of $J$.
Substitute \eqref{eq:model-matrix} in \eqref{eq:local-BL}. The factors
$(p-1)/p$ cancel.
\end{proof}

There is also a direct proof for unconditional test functions. On the
positive orthant, set $x_i=e^{u_i}$. The probability measure conditioned
on that orthant becomes a measure with potential
\[
 U(u)=V(e^{u_1},\ldots,e^{u_n})-\sum_i u_i
\]
up to an additive constant. If $D=\diag(x_1,\ldots,x_n)$, then
\[
 \nabla^2U
 =D\left(\nabla^2V+\diag\left(\frac{\partial_iV}{x_i}\right)\right)D.
\]
For $F(u)=f(e^u)$, we have $\nabla F=D\nabla f$. The ordinary
Brascamp--Lieb inequality for $U$ is therefore exactly
\eqref{eq:intro-mult-BL} on this orthant. If $f$ is unconditional, its
means and second moments are the same in every orthant, which proves the
full inequality. For a merely even function, the orthant means need not
agree; this explains why the logarithmic change of variables alone is
insufficient.

\begin{example}\label{ex:coordinate-power}
For $V_r(x)=\frac1r\sum_i|x_i|^r$, $r>1$,
\[
 \nabla^2V_r+\diag\left(\frac{\partial_iV_r}{x_i}\right)
 =r\,\diag(|x_i|^{r-2}).
\]
Thus the inequality takes the form
\begin{equation}\label{eq:coordinate-power}
 \Var_\mu(f)\le\frac1r\int\sum_i|x_i|^{2-r}(\partial_i f)^2\,d\mu,
 \qquad d\mu\propto e^{-\sum_i|x_i|^r/r}\,dx.
\end{equation}
For unconditional functions this follows for every $r>1$ from the
preceding argument. For $r\ge2$, the even-function inequality is also a
consequence of the strong Brascamp--Lieb inequality in \cite{CKLR}.
At $r=2$ it is the familiar Gaussian inequality
$\Var_\gamma(f)\le\frac12\int|\nabla f|^2\,d\gamma$ for even functions.
\end{example}

\section{An \texorpdfstring{$L^p$}{Lp} Fradelizi--Meyer theorem and radial cost inequalities}
\label{sec:p-FM}

The next argument starts from a weighted volume-product inequality.
It is independent of the Steiner condition on the cost.

\subsection{The set hypothesis and a polar-coordinate lemma}

Let $p>1$, and let $V:\R^n\to[0,\infty)$ be even, strictly convex,
$C^1$, and $p$-homogeneous, with $V(x)>0$ for $x\ne0$. Write
\[
 L=\{V\le1/p\},\qquad \|x\|=(pV(x))^{1/p}.
\]
Then $\|\cdot\|$ is a norm. These assumptions ensure that
$\nabla V$ is a bijection with inverse $\nabla V^*$, where $V^*$ is the
Legendre transform. The inverse is $1/(p-1)$-homogeneous.

We assume that every symmetric convex body $K$ satisfies
\begin{equation}\label{eq:set-santalo}
 |K|\,|\nabla V^*(K^\circ)|^{p-1}\le |L|^p.
\end{equation}
Here $\nabla V^*(K^\circ)$ denotes the image of the set, which need not
be convex. An assumption for convex bodies is enough: replacing a
symmetric bounded star-shaped set by its closed convex hull preserves
its polar and increases its volume.

For any nonnegative measurable $\rho$, polar coordinates give
\begin{equation}\label{eq:polar-normalization}
 \int_{\R^n}\rho(pV(x))\,dx
 =n|L|\int_0^\infty t^{n-1}\rho(t^p)\,dt.
\end{equation}
Indeed, setting $x=r\theta$ and then
$t=(pV(\theta))^{1/p}r$ gives the factor
$(pV(\theta))^{-n/p}$, whose integral over the sphere is $n|L|$.

\subsection{The \texorpdfstring{$L^p$}{Lp} Fradelizi--Meyer theorem}

\begin{theorem}\label{thm:FM}
Suppose \eqref{eq:set-santalo} holds. Let $\rho$ be nonnegative and
measurable, and let $f,g$ be nonnegative integrable functions, with $f$
even. If
\begin{equation}\label{eq:FM-hyp}
 f(x)g(y)^{p-1}\le\rho(\ip{x}{\nabla V(y)})^p
 \quad\text{whenever }\ip{x}{\nabla V(y)}>0,
\end{equation}
then
\begin{equation}\label{eq:FM}
 \int f\left(\int g\right)^{p-1}
 \le\left(\int\rho(pV(x))\,dx\right)^p.
\end{equation}
\end{theorem}

\begin{proof}
By truncating $f$ and $g$ from above and restricting them to increasing
centered balls, it is enough to prove the inequality for bounded,
compactly supported functions. We may assume that both integrals are
positive and that
\[
 0<I:=\int_0^\infty t^{n-1}\rho(t^p)\,dt<\infty.
\]
If $I=0$, replace $\rho(s)$ by $\rho(s)+\varepsilon e^{-s}$ and let
$\varepsilon\downarrow0$; if $I=\infty$, there is nothing to prove.

Define bounded star-shaped sets $A,B$ through their radial functions:
\[
 r_A(\theta)^n=\int_0^\infty r^{n-1}f(r\theta)\,dr,
 \qquad
 r_B(\eta)^n=\int_0^\infty s^{n-1}g(s\eta)\,ds.
\]
The set $A$ is symmetric, and
\begin{equation}\label{eq:AB-volumes}
 |A|=\frac1n\int f,\qquad |B|=\frac1n\int g.
\end{equation}

Fix $\theta,\eta\in\Sph^{n-1}$ with
$a=\ip{\theta}{\nabla V(\eta)}>0$. The homogeneity of $\nabla V$ turns
\eqref{eq:FM-hyp} into
\[
 \bigl(r^{n-1}f(r\theta)\bigr)^{1/p}
 \bigl(s^{n-1}g(s\eta)\bigr)^{(p-1)/p}
 \le u^{n-1}\rho(au^p),\qquad u=r^{1/p}s^{(p-1)/p}.
\]
The one-dimensional multiplicative Pr\'ekopa--Leindler inequality gives
\[
 r_A(\theta)^{n/p}r_B(\eta)^{n(p-1)/p}
 \le\int_0^\infty u^{n-1}\rho(au^p)\,du
 =a^{-n/p}I.
\]
Equivalently,
\[
 r_A(\theta)\,r_B(\eta)^{p-1}
 \ip{\theta}{\nabla V(\eta)}\le I^{p/n}.
\]
By homogeneity, this means
\[
 \ip{x}{\nabla V(y)}\le I^{p/n}\qquad(x\in A,\ y\in B).
\]
For nonpositive scalar products the same inequality is automatic. Hence
\[
 \nabla V(B)\subset I^{p/n}A^\circ,
 \qquad
 B\subset I^{p/(n(p-1))}\nabla V^*(A^\circ).
\]

Let $K$ be the closed convex hull of $A$. It is a symmetric convex body,
since $A$ is bounded and has positive volume. Also $K^\circ=A^\circ$ and
$|A|\le|K|$. Applying \eqref{eq:set-santalo}, we get
\[
 |A|\,|B|^{p-1}
 \le I^p|K|\,|\nabla V^*(K^\circ)|^{p-1}
 \le I^p|L|^p.
\]
Now use \eqref{eq:AB-volumes} and \eqref{eq:polar-normalization}.
Monotone convergence removes the truncations of $f$ and $g$.
\end{proof}

\begin{corollary}[Set--function equivalence]\label{cor:FM-equivalence}
For a fixed $V$ as above, \eqref{eq:set-santalo} for all symmetric
convex bodies is equivalent to the assertion of Theorem~\ref{thm:FM}
for every nonnegative measurable profile $\rho$.
\end{corollary}
\begin{proof}
Only the converse remains. Set $\rho=\mathbf1_{[0,1]}$,
$f=\mathbf1_K$ and $g=\mathbf1_{\nabla V^*(K^\circ)}$.
Whenever both indicators are nonzero,
$\ip{x}{\nabla V(y)}\le1$, so \eqref{eq:FM-hyp} holds.
Also $\int\rho(pV)=|L|$. Thus \eqref{eq:FM} is precisely
\eqref{eq:set-santalo}.
\end{proof}

\begin{proof}[Proof of Theorem~\ref{thm:intro-FM}]
For $V(x)=\|x\|_r^p/p$ with $p\ge r\ge2$, the weighted functional
inequality in \cite[Corollary~1.5]{CKLR}, together with its set
formulation \cite[Corollary~6.4]{CKLR}, gives
\eqref{eq:set-santalo}. Here $L=B_r^n$, and
\[
 \nabla V(y)=\|y\|_r^{p-r}
 (|y_1|^{r-2}y_1,\ldots,|y_n|^{r-2}y_n)
\]
for $y\ne0$, with $\nabla V(0)=0$. Theorem~\ref{thm:FM} applies.
\end{proof}

\subsection{Cost formulation}

Continue to assume \eqref{eq:set-santalo}, with $V(x)=\|x\|^p/p$.
Let $w:[0,\infty)\to\R$ be continuous and nondecreasing, and assume
\begin{equation}\label{eq:w-assumption}
 t\longmapsto w(e^t)\quad\text{is convex},
 \qquad 0<\int e^{-w(\|x\|)}\,dx<\infty.
\end{equation}
Set
\[
 W(x)=w(\|x\|),\qquad u(s)=p\,w(s^{1/p}).
\]
Define the cost on all pairs by
\begin{equation}\label{eq:radial-cost}
 c_u(x,y)=u\bigl(\ip{x}{\nabla V(y)}_+\bigr).
\end{equation}
The positive part gives the value $u(0)=pw(0)$ when the scalar product is
nonpositive. Only the positive-inner-product region is needed to apply
Theorem~\ref{thm:FM}.

\begin{proposition}\label{prop:radial-santalo}
For every even measurable $\phi$ with $0<\int e^{-\phi}<\infty$,
\begin{equation}\label{eq:radial-santalo}
 \int e^{-\phi}
 \left(\int e^{-\phi^{c_u}/(p-1)}\right)^{p-1}
 \le\left(\int e^{-W}\right)^p.
\end{equation}
Moreover, $W^{c_u}=(p-1)W$, so equality holds at $W$.
\end{proposition}

\begin{proof}
Apply Theorem~\ref{thm:FM} with
\[
 \rho(s)=e^{-u(s)/p},\qquad f=e^{-\phi},\qquad
 g=e^{-\phi^{c_u}/(p-1)}.
\]
The hypothesis follows from the definition of the cost transform. If
the integrability of $g$ is not known, first truncate it and then use
monotone convergence. Since $u(pV(x))/p=W(x)$, this proves
\eqref{eq:radial-santalo}.

For self-duality, norm duality gives
\[
 \ip{x}{\nabla V(y)}_+\le\|x\|\,\|y\|^{p-1}.
\]
The convexity of $t\mapsto w(e^t)$ yields, for nonzero $x,y$,
\[
 p\,w(\|x\|^{1/p}\|y\|^{(p-1)/p})
 \le w(\|x\|)+(p-1)w(\|y\|).
\]
Together with monotonicity, this gives
$c_u(x,y)\le W(x)+(p-1)W(y)$. Continuity gives the same inequality when
one vector vanishes. Equality holds at $x=y$, because
$\ip{y}{\nabla V(y)}=pV(y)=\|y\|^p$. This proves the claim.
\end{proof}

\subsection{Differentiating the radial cost inequality}

Assume now that $V,W,u$ are $C^2$ away from the origin. We use even
perturbations for which the contact point is unique, the maximizing
branch is differentiable, and the two integral derivatives are justified,
as in Section~\ref{sec:differentiation-BL}. Here the cost inequality holds for all even measurable potentials, so
there is no convexity restriction on the perturbed potential.

Near $x=y\ne0$, the scalar product is positive and
\[
 D^2_{xx}c_u(x,y)
 =u''(\ip{x}{\nabla V(y)})\,
     \nabla V(y)\otimes\nabla V(y).
\]
Since $W=u(pV)/p$,
\[
 \nabla^2W=p\,u''(pV)\nabla V\otimes\nabla V+u'(pV)\nabla^2V.
\]
It follows that
\begin{equation}\label{eq:radial-matrix}
 \nabla^2W-D^2_{xx}c_u(x,x)
 =\frac{p-1}{p}\nabla^2W+\frac1p u'(pV)\nabla^2V.
\end{equation}
Also,
\[
 u'(pV(x))=\frac{w'(\|x\|)}{\|x\|^{p-1}}.
\]
Applying Proposition~\ref{prop:local-BL} with $\alpha=1/(p-1)$ gives
\begin{equation}\label{eq:radial-BL}
 \Var_\mu(f)\le
 \int\left\langle
 \left(\nabla^2W+
 \frac{w'(\|x\|)}{(p-1)\|x\|^{p-1}}\nabla^2V\right)^{-1}
 \nabla f,\nabla f\right\rangle\,d\mu,
 \qquad d\mu=\frac{e^{-W}}{\int e^{-W}}\,dx,
\end{equation}
for these even first variations, provided the displayed matrix is
positive definite. As before, the inequality extends to the closure of
this class in its $L^2$ and energy norms.

In the Euclidean case $p=2$, $V(x)=|x|^2/2$, this is
\begin{equation}\label{eq:Euclidean-radial-BL}
 \Var_\mu(f)\le
 \int\left\langle
 \left(\nabla^2W+\frac{w'(|x|)}{|x|}\Id\right)^{-1}
 \nabla f,\nabla f\right\rangle\,d\mu.
\end{equation}
This recovers the radial weighted inequality of Cordero-Erausquin and
Rotem \cite{CR} under the regularity assumptions used here. Their work
also treats measures with an additional even convex perturbation of the
radial potential.

For $w(t)=t^p/p$, we have $W=V$ and
$w'(\|x\|)/\|x\|^{p-1}=1$. Thus \eqref{eq:radial-BL} reduces to
\[
 \Var_\mu(f)\le\frac{p-1}{p}
 \int\ip{(\nabla^2V)^{-1}\nabla f}{\nabla f}\,d\mu,
 \qquad d\mu\propto e^{-V}\,dx.
\]

\subsection{Dilation equality and transport}

Write $R=\|x\|$, and denote the matrix in \eqref{eq:radial-BL} by
\begin{equation}\label{eq:MW}
 M_W=\nabla^2W+
 \frac{w'(R)}{(p-1)R^{p-1}}\nabla^2V.
\end{equation}
The costs \eqref{eq:radial-cost} satisfy
$c_u(e^t x,e^{-t/(p-1)}y)=c_u(x,y)$. Consequently, for
$\phi_t(x)=W(e^t x)$,
\[
 \phi_t^{c_u}(y)=(p-1)W(e^{-t/(p-1)}y).
\]
The two changes of variables show that equality in
\eqref{eq:radial-santalo} holds along this whole family. Whenever its
second variation is justified, its tangent
\begin{equation}\label{eq:dilation-score}
 G(x)=\ip{x}{\nabla W(x)}=Rw'(R)
\end{equation}
is therefore an equality direction for \eqref{eq:radial-BL}.
Euler's identity gives a useful pointwise version:
\begin{equation}\label{eq:score-matrix}
 M_Wx=\nabla^2W\,x+\nabla W=\nabla G.
\end{equation}

For probability measures $\lambda,\mu$, write
$D_{\mathrm{KL}}(\lambda\|\mu)=\int\log(d\lambda/d\mu)\,d\lambda$
when $\lambda\ll\mu$, and $+\infty$ otherwise.

\begin{proposition}\label{prop:radial-transport}
Under the assumptions of Proposition~\ref{prop:radial-santalo}, put
$d\mu=Z^{-1}e^{-W}dx$ and
\[
 \mathsf D(x,y)=W(x)+(p-1)W(y)-c_u(x,y).
\]
Then $\mathsf D\ge0$, and for even probability measures $\lambda,\eta$,
\begin{equation}\label{eq:radial-transport}
 \inf_{\pi\in\Pi(\lambda,\eta)}\int\mathsf D\,d\pi
 \le D_{\mathrm{KL}}(\lambda\|\mu)
       +(p-1)D_{\mathrm{KL}}(\eta\|\mu).
\end{equation}
\end{proposition}
\begin{proof}
Nonnegativity is the pointwise inequality in the self-duality proof.
The cost $\mathsf D$ is continuous and invariant under
$(x,y)\mapsto(-x,-y)$. In Kantorovich duality, replace bounded
continuous admissible functions $a,b$ by their even parts. Their
integrals do not change, and $a(x)+b(y)\le\mathsf D(x,y)$ is preserved.
Take $\phi=W-a$ and $\psi=(p-1)W-b$. Then
$\phi+\psi\ge c_u$, so $\psi\ge\phi^{c_u}$ and
Proposition~\ref{prop:radial-santalo} yields
\[
 \log\int e^a\,d\mu+(p-1)\log\int e^{b/(p-1)}\,d\mu\le0.
\]
The entropy variational inequality applied to $a$ and $b/(p-1)$
bounds $\int a\,d\lambda+\int b\,d\eta$ by the right-hand side of
\eqref{eq:radial-transport}. Taking the dual supremum proves the claim.
\end{proof}

This is the usual functional-to-transport duality in the present
asymmetric normalization; compare \cite{Fathi,FGSZ,LMR}. Locally,
$D^2_{xx}\mathsf D(x,x)=\frac{p-1}{p}M_W(x)$, so the global transport
cost and the variance estimate determine the same quadratic form,
up to the normalization already present in the second variation.

\section{Examples and sharpness}\label{sec:examples}

In the examples below, a variance inequality obtained from
\eqref{eq:radial-BL} is asserted for its admissible even first variations
and their closure in $L^2$ and in the displayed energy norm. Values of
the coefficient matrices on coordinate hyperplanes are immaterial.
For singular weights, almost-everywhere invertibility alone does not
assert density of an arbitrary smooth test class. The equality
functions below arise from the dilation family just described and
have finite variance and energy.

\subsection{Elliptic Gaussian and bilinear costs}

The diagonal restriction in Corollary~\ref{cor:quadratic} belongs to
its coordinate proof, rather than to the inequality. In fact, if
$A$ is any invertible real matrix and $c_A(x,y)=\ip{x}{Ay}$, then
$\phi^{c_A}(y)=\phi^*(Ay)$ and the classical functional Santal\'o
inequality gives
\[
 \int e^{-\phi(x)}dx\int e^{-\phi^{c_A}(y)}dy
 \le\frac{(2\pi)^n}{|\det A|}
\]
for every even measurable $\phi$ for which the primal integral is
positive and finite. For $A=A^T>0$, the potential
$V(x)=\frac12\ip{Ax}{x}$ is self-dual. Its second variation gives
\[
 \Var_{\gamma_A}(f)\le\frac12\int\ip{A^{-1}\nabla f}{\nabla f}
 \,d\gamma_A,\qquad d\gamma_A\propto e^{-\ip{Ax}{x}/2}dx.
\]
This is also the usual even Gaussian Poincar\'e inequality after a
linear change of variables. Centered quadratic functions attain equality.

\subsection{Anisotropic power densities}

\begin{proposition}\label{prop:power-example}
Let $r\ge2$, $q>0$, $R=\|x\|_r$ and
$d\mu_{r,q}=Z^{-1}e^{-R^q/q}dx$. For the admissible even functions
specified above,
\begin{equation}\label{eq:power-example}
 \Var_{\mu_{r,q}}(f)\le
 \int\left[\frac{R^{r-q}}r\sum_i|x_i|^{2-r}(\partial_i f)^2
 +\frac{r-q}{rq}R^{-q}\ip{x}{\nabla f}^{2}\right]d\mu_{r,q}.
\end{equation}
The bound has equality for $f=R^q$.
\end{proposition}
\begin{proof}
Use $p=r$, $V=R^r/r$ and $w(t)=t^q/q$ in Section~\ref{sec:p-FM}.
The set hypothesis follows from \cite{CKLR}. The function
$w(e^t)=e^{qt}/q$ is strictly convex for every $q>0$.
On the open orthants, put
\[
 D=\diag(|x_i|^{r-2}),\qquad a=(|x_i|^{r-2}x_i)_i.
\]
Direct differentiation gives
\[
 M_W=rR^{q-r}D+(q-r)R^{q-2r}a\otimes a.
\]
Since $D^{-1}a=x$ and $\ip{D^{-1}a}{a}=R^r$, its inverse is
\[
 M_W^{-1}=\frac{R^{r-q}}rD^{-1}
             +\frac{r-q}{rq}R^{-q}x\otimes x.
\]
The matrix is positive definite: after conjugation by $D^{-1/2}$
and removal of the positive factor $R^{q-r}$, its eigenvalues are
$r$ on the orthogonal complement of $D^{-1/2}a$ and $q$ in that
direction. Formula \eqref{eq:radial-BL} proves the inequality.

Polar coordinates give $R^q/q\sim\operatorname{Gamma}(n/q,1)$,
where the second parameter is the scale. Hence
$\int R^q\,d\mu_{r,q}=n$ and $\Var_{\mu_{r,q}}(R^q)=nq$.
Here $G=R^q$, and \eqref{eq:score-matrix} shows that its energy is
$\int\ip{x}{\nabla G}\,d\mu_{r,q}=q\int R^q\,d\mu_{r,q}=nq$.
\end{proof}

For $q=r$ this is \eqref{eq:coordinate-power}. For $r=2$ it reads
\[
 \Var_\mu(f)\le\int |x|^{2-q}
 \left(\frac{(\partial_{|x|}f)^2}{q}
       +\frac{|\nabla_{\!\mathrm{tan}}f|^2}{2}\right)d\mu,
\]
where $\nabla_{\!\mathrm{tan}}f$ is the Euclidean tangential component
of $\nabla f$; this case is contained in \cite{CR}.
For $0<q<1$, $W$ is not convex, but the matrix above remains positive.

When $q\ge r$, the matrix estimate refines the scalar strong
Brascamp--Lieb estimate with coefficient $(q-1)/q$. Indeed,
\begin{equation}\label{eq:power-comparison}
 M_W-\frac q{q-1}\nabla^2W
 =\frac{q-r}{q-1}R^{q-2r}(R^rD-a\otimes a)\succeq0.
\end{equation}
The last inequality is weighted Cauchy--Schwarz. If $q>r$ and
$n\ge2$, the difference is positive in transverse directions, with
kernel spanned by $x$ on each open orthant.

\subsection{Anisotropic Cauchy densities}

\begin{proposition}\label{prop:cauchy-example}
Let $r\ge2$, $\beta>n/r$, $R=\|x\|_r$ and
$d\mu_{r,\beta}=Z^{-1}(1+R^r)^{-\beta}dx$. Then
\begin{equation}\label{eq:cauchy-example}
 \Var_{\mu_{r,\beta}}(f)\le\frac1{\beta r^2}
 \int(1+R^r)\left[\sum_i|x_i|^{2-r}(\partial_i f)^2
                   +\ip{x}{\nabla f}^{2}\right]d\mu_{r,\beta}
\end{equation}
for the same admissible class. Equality holds for
$f(x)=R^r/(1+R^r)$.
\end{proposition}
\begin{proof}
Take $p=r$, $V=R^r/r$ and $w(t)=\beta\log(1+t^r)$.
The integrability condition is exactly $\beta>n/r$, and
$w(e^t)=\beta\log(1+e^{rt})$ is strictly convex. With $D,a$ as in
the preceding proof and $S=R^r$,
\[
 M_W=\beta r^2\left(\frac{D}{1+S}
              -\frac{a\otimes a}{(1+S)^2}\right),\qquad
 M_W^{-1}=\frac{1+S}{\beta r^2}(D^{-1}+x\otimes x).
\]
This proves \eqref{eq:cauchy-example} by \eqref{eq:radial-BL}.
For $U=S/(1+S)$, polar coordinates show
$U\sim\operatorname{Beta}(\alpha,\beta-\alpha)$ with $\alpha=n/r$.
Its variance is $\alpha(\beta-\alpha)/(\beta^2(\beta+1))$.
As $G=\beta rU$, \eqref{eq:score-matrix} gives the energy of $U$ as
$\beta^{-1}\int U(1-U)\,d\mu_{r,\beta}$, the same number.
\end{proof}

At $r=2$, \eqref{eq:cauchy-example} is the Cauchy inequality of
\cite{CR}. The equality function is bounded, so no additional
polynomial-moment assumption is needed beyond $\beta>n/r$.

\subsection{The obstruction below exponent two}

\begin{proposition}\label{prop:subcritical-obstruction}
For $n\ge2$ and $1<r<2$, \eqref{eq:coordinate-power} fails on even
functions of finite energy. It also fails for some smooth even functions.
\end{proposition}
\begin{proof}
Let $d\nu_r(t)=z_r^{-1}e^{-|t|^r/r}dt$ and
$h(t)=\operatorname{sgn}(t)|t|^{r-1}$. Integration by parts gives
\[
 \int |t|^{2r-2}d\nu_r=(r-1)\int |t|^{r-2}d\nu_r,
\]
and therefore
\[
 \frac1r\int |t|^{2-r}h'(t)^2d\nu_r
 =\frac{r-1}{r}\int h^2d\nu_r.
\]
Both integrals are finite because $r>1$.
For the product measure $\mu_r=\nu_r^{\otimes n}$, take
$f(x)=h(x_1)h(x_2)$. It is even and has mean zero. Independence gives
\[
 \frac1r\int\sum_i|x_i|^{2-r}(\partial_i f)^2d\mu_r
 =\frac{2(r-1)}r\Var_{\mu_r}(f)<\Var_{\mu_r}(f).
\]
To obtain smooth examples, replace $h$ near zero by a smooth odd
interpolation with derivative $O(\varepsilon^{r-2})$ on
$|t|\le2\varepsilon$. The change in weighted energy there is
$O(\varepsilon^{r-1})$. Smooth truncation at infinity converges in
both $L^2(\nu_r)$ and energy. Products of these approximations
preserve evenness and the strict violation.
\end{proof}

This illustrates the difference between central symmetry and
unconditionality: a product of two odd coordinate modes is even but
not unconditional. At $r=2$ the same construction gives the Gaussian
equality mode $x_1x_2$.

\section{Spherical Blaschke--Santal\'o and logarithmic transport}
\label{sec:spherical-BS}

Let $n\ge2$, and equip $\Sph^{n-1}$ with normalized surface measure
$\sigma$. We use the cost
\[
 c(x,y)=
 \begin{cases}
 \log\ip{x}{y},&\ip{x}{y}>0,\\
 -\infty,&\ip{x}{y}\le0,
 \end{cases}
\]
and the transform
\[
 f^c(y)=\sup_{\ip{x}{y}>0}\{\log\ip{x}{y}-f(x)\}.
\]

\begin{theorem}\label{thm:spherical}
Let $f,g:\Sph^{n-1}\to\R\cup\{+\infty\}$ be measurable, with $f$ even,
and suppose
\[
 f(x)+g(y)\ge\log\ip{x}{y}
 \qquad\text{whenever }\ip{x}{y}>0.
\]
If the two integrals are finite, then
\begin{equation}\label{eq:spherical}
 \int e^{-nf}\,d\sigma\int e^{-ng}\,d\sigma\le1.
\end{equation}
In particular,
\begin{equation}\label{eq:spherical-dual}
 \int e^{-nf}\,d\sigma\int e^{-nf^c}\,d\sigma\le1
\end{equation}
whenever the first integral is positive and finite.
\end{theorem}

\begin{proof}
Let $A,B$ be the measurable star-shaped sets with radial functions
$e^{-f}$ and $e^{-g}$. The hypothesis says that
$\ip{x}{y}\le1$ for $x\in A$ and $y\in B$; nonpositive scalar products
require no condition. Thus $B\subset A^\circ$.

We may assume both sets have positive volume. Since $A$ is symmetric,
choosing $n$ linearly independent points of $B$ shows that $A$ is bounded:
every $x\in A$ satisfies $|\ip{x}{y_i}|\le1$ for these points.
Its closed convex hull $K$ is therefore a symmetric convex body, and
$K^\circ=A^\circ$. The classical inequality gives
\[
 |A|\,|B|\le|K|\,|K^\circ|\le|B_2^n|^2.
\]
Polar coordinates yield
\[
 |A|=|B_2^n|\int e^{-nf}\,d\sigma,\qquad
 |B|=|B_2^n|\int e^{-ng}\,d\sigma,
\]
which proves \eqref{eq:spherical}. A zero integral makes the conclusion
immediate. For the transform version, apply the pair inequality to
$g_m=\max\{f^c,-m\}$ and then let $m\to\infty$. This also proves the
integrability of $e^{-nf^c}$.
\end{proof}

For continuous finite $f$, equality in \eqref{eq:spherical-dual} holds
exactly when $e^{-f}$ is the radial function of a centered ellipsoid.
Indeed, equality forces $|A|=|K|$ and equality in the classical volume
product; continuity then identifies $A$ with $K$. Equivalently,
\[
 f(x)=\tfrac12\log\ip{Ax}{x}
\]
for a positive definite matrix $A$. Scaling is included in this formula.

\subsection{Gradient, Hessian and contact map}

Fix $y$ and write $a=\ip{x}{y}>0$. The spherical derivatives of the cost
in its first variable are
\begin{equation}\label{eq:spherical-gradient}
 \gradS^x c(x,y)=\frac{y}{a}-x
\end{equation}
and, for $u,v\in T_x\Sph^{n-1}$,
\begin{equation}\label{eq:spherical-Hessian}
 \operatorname{Hess}_S^x c(x,y)[u,v]
 =-\ip{u}{v}-\frac{\ip{u}{y}\ip{v}{y}}{a^2}.
\end{equation}
The first formula is the tangential projection of $y/a$. For the second,
use $\HessS\ip{x}{y}=-\ip{x}{y}\Id$ and differentiate the logarithm.

For the change-of-variables calculations, assume $f$ is smooth and its
contact map is a diffeomorphism. We also assume that the contact Hessian
\[
 M_f=\Id+\HessS f+\gradS f\otimes\gradS f
\]
is positive definite on each tangent space. These assumptions hold when
$e^{-f}$ is the radial function of a smooth strictly convex body with
positive curvature. The map $T$ below sends a radial direction to its
outward normal, so that the maximizing point in the definition of
$f^c(T(x))$ is $x$.

At a contact pair $y=T(x)$,
\[
 f(x)+f^c(T(x))=\log\ip{x}{T(x)}.
\]
Equation \eqref{eq:spherical-gradient} gives
$\gradS f=T/a-x$. Since $|T|=1$, we obtain
\begin{equation}\label{eq:spherical-map}
 T(x)=\frac{x+\gradS f(x)}{\sqrt{1+|\gradS f(x)|^2}},
 \qquad a=\ip{x}{T(x)}=\frac1{\sqrt{1+|\gradS f(x)|^2}}.
\end{equation}

To compute its Jacobian, put $g=\gradS f$ and let $P_T$ be the orthogonal
projection onto $T_{T(x)}\Sph^{n-1}$. For a tangent vector $v$,
\[
 DT(x)[v]=aP_T\bigl(v+\HessS f[v]+\ip{g}{v}g\bigr).
\]
Indeed, the ambient derivative of $g$ is
$\HessS f[v]-\ip{g}{v}x$, and $P_T(x+g)=0$ gives the displayed form.
The projection from $T_x\Sph^{n-1}$ to $T_{T(x)}\Sph^{n-1}$ has absolute
determinant $a$. As the tangent spaces have dimension $n-1$,
\begin{equation}\label{eq:spherical-Jacobian}
 J_T(x)=a^n\det M_f(x).
\end{equation}

Suppose now that $f$ is a smooth extremizer of
\eqref{eq:spherical-dual}. Write
\[
 A=\int e^{-nf}\,d\sigma,\qquad B=\int e^{-nf^c}\,d\sigma.
\]
The first variation gives
\[
 T_\#\left(\frac{e^{-nf}}{A}\,d\sigma\right)
 =\frac{e^{-nf^c}}{B}\,d\sigma.
\]
Even variations suffice, because the two measures and the contact map
have the corresponding antipodal symmetry. Therefore
\[
 \frac{e^{-nf(x)}}{A}
 =\frac{e^{-nf^c(T(x))}}{B}J_T(x).
\]
Substitute the contact identity and
\eqref{eq:spherical-Jacobian}. The factors $a^n$ cancel, leaving
\begin{equation}\label{eq:spherical-MA}
 \det\bigl(\Id+\HessS f+\gradS f\otimes\gradS f\bigr)
 =\frac BA e^{-2nf}.
\end{equation}

\subsection{Second variation and the even Poincar\'e inequality}

Let $h$ be smooth and even, and define
\[
 F(t)=\int e^{-nth}\,d\sigma\int e^{-n(th)^c}\,d\sigma.
\]
At $t=0$, the maximizing point in
\[
 (th)^c(y)=\sup_{\ip{x}{y}>0}\{\log\ip{x}{y}-th(x)\}
\]
is $x=y$. The Hessian of the cost there is $-\Id$. Differentiating the
contact condition gives $\dot x_0(y)=-\gradS h(y)$, and the envelope
calculation yields
\[
 \left.\frac{d}{dt}\right|_0(th)^c(y)=-h(y),\qquad
 \left.\frac{d^2}{dt^2}\right|_0(th)^c(y)=|\gradS h(y)|^2.
\]
Thus
\[
 (th)^c=-th+\frac{t^2}{2}|\gradS h|^2+o(t^2).
\]
For small $t$ these expansions are uniform on the compact sphere: the
maximizer remains near $x=y$, and the contact Hessian remains negative
definite.

If $A(t)=\int e^{-nth}\,d\sigma$ and
$B(t)=\int e^{-n(th)^c}\,d\sigma$, then $A(0)=B(0)=1$ and
\[
 A'(0)=-n\int h\,d\sigma,\qquad B'(0)=n\int h\,d\sigma,
\]
\[
 A''(0)=n^2\int h^2\,d\sigma,\qquad
 B''(0)=n^2\int h^2\,d\sigma-n\int|\gradS h|^2\,d\sigma.
\]
Consequently,
\begin{equation}\label{eq:spherical-second}
 F''(0)=2n^2\Var_\sigma(h)-n\int|\gradS h|^2\,d\sigma.
\end{equation}
Since $F(t)\le F(0)=1$, we recover the sharp even Poincar\'e inequality
\begin{equation}\label{eq:spherical-Poincare}
 \int|\gradS h|^2\,d\sigma\ge2n\Var_\sigma(h).
\end{equation}
The constant is attained by the restrictions of trace-free quadratic
forms to the sphere.

\begin{proposition}[Deficit transverse to ellipsoids]\label{prop:sphere-local-deficit}
Set
\[
 \delta(f)=-\log\left(\int e^{-nf}d\sigma
                           \int e^{-nf^c}d\sigma\right).
\]
For smooth even $h$,
\[
 \left.\frac{d^2}{dt^2}\right|_0\delta(th)
 =n\int|\gradS h|^2d\sigma-2n^2\Var_\sigma(h).
\]
Its kernel consists of constants and degree-two spherical harmonics.
If $h$ is orthogonal to both these spaces, then
\begin{equation}\label{eq:sphere-transverse}
 \lim_{t\to0}\frac{\delta(th)}{t^2}
 \ge n(n+4)\|h\|_{L^2(\sigma)}^2.
\end{equation}
\end{proposition}
\begin{proof}
Use \eqref{eq:spherical-second} and $F(0)=1$, $F'(0)=0$.
The eigenvalues of $-\Delta_S$ are $k(k+n-2)$; evenness allows only
even $k$. The zero modes of the displayed quadratic form are $k=0,2$.
After their removal, the smallest eigenvalue is $4(n+2)$, giving
\eqref{eq:sphere-transverse}. Finally, differentiating
$\frac12\log\ip{(\Id+tA)x}{x}$ at zero identifies the $k=0,2$
modes with infinitesimal ellipsoidal deformations.
\end{proof}

The assertion is a second-order estimate along a fixed smooth
perturbation. A uniform nonlinear stability theorem would additionally
require control of the remainder and a choice of ellipsoidal normalization.

\subsection{Entropy--logdet and logarithmic transport}

Return to a smooth even $f$ with a diffeomorphic contact map and
$M_f>0$, and normalize it by $\int e^{-nf}\,d\sigma=1$. Put
$d\mu=e^{-nf}\,d\sigma$. The contact identity and Jacobian formula give
\[
 \int e^{-nf^c}\,d\sigma
 =\int e^{-nf^c(T(x))}J_T(x)\,d\sigma(x)
 =\int e^{nf}\det M_f\,d\sigma.
\]
By \eqref{eq:spherical-dual}, this is at most one. Equivalently,
\begin{equation}\label{eq:spherical-logdet-integral}
 \int e^{2nf}\det M_f\,d\mu\le1.
\end{equation}
Jensen's inequality then gives
\begin{equation}\label{eq:spherical-logdet}
 \int\bigl(2nf+\log\det M_f\bigr)\,d\mu\le0.
\end{equation}

The Jensen loss admits an exact interpretation. Let
$B=\int e^{-nf^c}d\sigma$ and
$d\nu=B^{-1}e^{-nf^c}d\sigma$. Then
\begin{equation}\label{eq:exact-spherical-deficit}
 2\Ent_\sigma(\mu)-\int\log\det M_f\,d\mu
 =-\log B+D_{\mathrm{KL}}(T_\#\mu\|\nu).
\end{equation}
Here $\Ent_\sigma(\mu)=\int\log(d\mu/d\sigma)d\mu$.
Indeed, the contact identity and \eqref{eq:spherical-Jacobian} give
\[
 \frac{d(T_\#\mu)}{d\nu}(T(x))
 =\frac{B e^{-nf(x)+nf^c(T(x))}}{J_T(x)}
 =\frac{B e^{-2nf(x)}}{\det M_f(x)}.
\]
Integrate the logarithm against $\mu$ and use
$\Ent_\sigma(\mu)=-n\int f\,d\mu$.
Thus the left-hand side of \eqref{eq:exact-spherical-deficit} is the
sum of the Santal\'o deficit and the relative entropy measuring failure
of the extremal transport relation. For instance, Pinsker's inequality
bounds it below by
$-\log B+2\|T_\#\mu-\nu\|_{\mathrm{TV}}^2$, with
$\|\eta-\nu\|_{\mathrm{TV}}=\sup_A|\eta(A)-\nu(A)|$.

We finish with the transport formulation. Set
\[
 d(x,y)=
 \begin{cases}
 \log\dfrac1{\ip{x}{y}},&\ip{x}{y}>0,\\[0.3em]
 +\infty,&\ip{x}{y}\le0,
 \end{cases}
 \qquad
 K(\mu,\nu)=\inf_{\pi\in\Pi(\mu,\nu)}\int d(x,y)\,d\pi(x,y).
\]
Here $\Pi(\mu,\nu)$ denotes the couplings of the two probability measures.
Their entropy relative to $\sigma$ is
\[
 \Ent_\sigma(\mu)=\int\rho\log\rho\,d\sigma
 \quad\text{if }d\mu=\rho\,d\sigma,
\]
and is $+\infty$ otherwise.

\begin{theorem}[Fradelizi--Gozlan--Sadovsky--Zugmeyer]
\label{thm:spherical-transport}
For even probability measures $\mu,\nu$ on $\Sph^{n-1}$,
\begin{equation}\label{eq:spherical-transport}
 K(\mu,\nu)\le\frac1n\Ent_\sigma(\mu)+\frac1n\Ent_\sigma(\nu).
\end{equation}
\end{theorem}

This is \cite[Theorem~3.7]{FGSZ}, written on $\Sph^{n-1}$ rather than
$\Sph^n$. We include the short proof from spherical Blaschke--Santal\'o.

\begin{proof}
We may assume both entropies are finite. Kantorovich duality for the
nonnegative lower semicontinuous cost $d$ reduces the assertion to bounding
\[
 \int u\,d\mu+\int v\,d\nu
\]
for bounded continuous functions satisfying $u(x)+v(y)\le d(x,y)$.
Replace them by their even parts
\[
 u_e(x)=\frac{u(x)+u(-x)}2,\qquad
 v_e(y)=\frac{v(y)+v(-y)}2.
\]
The integrals do not change, and the constraint is preserved by averaging
it at $(x,y)$ and $(-x,-y)$. Applying Theorem~\ref{thm:spherical} to
$f=-u_e$ and $g=-v_e$ gives
\[
 \log\int e^{nu_e}\,d\sigma+\log\int e^{nv_e}\,d\sigma\le0.
\]
The entropy variational inequality says
\[
 \int\psi\,d\lambda\le
 \Ent_\sigma(\lambda)+\log\int e^\psi\,d\sigma.
\]
Use it with $\psi=nu_e$, $\lambda=\mu$, and then with $\psi=nv_e$,
$\lambda=\nu$. Adding gives the required bound on the dual expression.
Taking the supremum proves \eqref{eq:spherical-transport}.
\end{proof}

\section{Restrictions and dilation concavity}\label{sec:restriction}

The whole-space estimates above motivate a stronger question. Let
$W=w(R)$ and $M_W$ be as in \eqref{eq:MW}. For a symmetric convex
body $K$ of positive $\mu$-measure, let $\mu_K=\mu(K)^{-1}\mathbf1_K\mu$.
One may ask for which reference potentials the inequality
\begin{equation}\label{eq:restriction-question}
 \Var_{\mu_K}(f)\le
 \int\ip{M_W^{-1}\nabla f}{\nabla f}\,d\mu_K
\end{equation}
holds for every even smooth function of finite energy and every such $K$.
This is a question about the entire even spectrum. In the Euclidean
radial case, \cite{CR} gives the corresponding result, and more generally
allows an additional even convex perturbation of the reference potential.
Our whole-space cost inequality does not itself establish
\eqref{eq:restriction-question} in the anisotropic case.

The relevance of the dilation score $G=\ip{x}{\nabla W}$ is exact.
For $F_K(t)=\log\mu(e^tK)$, change variables to $K$ and differentiate
twice to obtain
\begin{equation}\label{eq:B-second}
 F_K''(0)=\Var_{\mu_K}(G)
             -\int\ip{x}{\nabla G}\,d\mu_K.
\end{equation}
The calculation is valid whenever the two derivatives are integrable;
for the power and Cauchy examples on bounded $K$, this follows directly
from their radial formulas. By \eqref{eq:score-matrix}, the right-hand
side is nonpositive if \eqref{eq:restriction-question} is available
just for $f=G$. Applying this at every dilate gives the scalar $(B)$
property. Conversely, that property gives precisely this one-function
variance estimate, without yielding the full restricted inequality.
For $W=R^q/q$, the one-function estimate is
\begin{equation}\label{eq:restricted-power-moment}
 \Var_{\mu_K}(R^q)\le q\int R^q\,d\mu_K.
\end{equation}

The general scalar $(B)$ assertion for even log-concave measures is
included in the recent OpenAI announcement
\cite{OpenAILogBM,OpenAICollection}. Accordingly, if that announced
result is used, \eqref{eq:restricted-power-moment} already follows
for $q\ge1$ and every norm $R$ by \eqref{eq:B-second}.
It does not supply \eqref{eq:restriction-question} for arbitrary even
$f$, nor does its log-concavity hypothesis cover the anisotropic
$0<q<1$ or Cauchy examples. These distinctions specify the additional
content sought from a restriction theorem. As noted in the introduction,
none of our proofs relies on the announced result.

\end{document}